\documentclass[11pt, reqno]{amsart}
\usepackage{indentfirst, amssymb, amsmath, amsthm, mathrsfs, setspace, indentfirst, enumerate,  mathrsfs, amsmath}
\usepackage[bookmarksnumbered, colorlinks, plainpages]{hyperref}
\usepackage{mathrsfs}

\newtheorem*{theo4.A}{Theorem 4.A}
\newtheorem*{theo4.B}{Theorem 4.B}

\newtheorem*{theo1.A}{Theorem 1.A}
\newtheorem*{theo1.B}{Theorem 1.B}
\newtheorem*{theo2.A}{Theorem 2.A}
\newtheorem*{theo2.B}{Theorem 2.B}
\newtheorem*{theo2.C}{Theorem 2.C}
\newtheorem*{theo2.D}{Theorem 2.D}
\newtheorem*{theo2.E}{Theorem 2.E}
\newtheorem*{theo2.F}{Theorem 2.F}

\newtheorem*{theo3.A}{Theorem 3.A}
\newtheorem*{theo3.B}{Theorem 3.B}
\newtheorem*{ques3.A}{Question 3.A}

\newtheorem*{cor A}{Corollary A}
\newtheorem*{cor B}{Corollary B}
\newtheorem*{exm2A}{Example 2.A}

\newtheorem{theo}{Theorem}[section]
\newtheorem{lem}{Lemma}[section]
\newtheorem{cor}{Corollary}[section]

\newtheorem{exm}{Example}[section]

\newtheorem{rem}{Remark}[section]

\newcommand{\ol}{\overline}
\newcommand{\be}{\begin{equation}}
\newcommand{\ee}{\end{equation}}
\newcommand{\beas}{\begin{eqnarray*}}
\newcommand{\eeas}{\end{eqnarray*}}
\newcommand{\bea}{\begin{eqnarray}}
\newcommand{\eea}{\end{eqnarray}}

\numberwithin{equation}{section}
\begin{document}
\title[P\MakeLowercase {artial sharing and cross sharing}......]{\LARGE
P\Large\MakeLowercase {artial sharing and cross sharing of entire function with its derivative}}
\date{}
\author[S. M\MakeLowercase{ajumder}, N. S\MakeLowercase{arkar} \MakeLowercase{and}  D. P\MakeLowercase{ramanik}]
{S\MakeLowercase{ujoy} M\MakeLowercase{ajumder}, N\MakeLowercase{abadwip} S\MakeLowercase{arkar}$^*$ \MakeLowercase{and} D\MakeLowercase{ebabrata} P\MakeLowercase{ramanik}}
\address{Department of Mathematics, Raiganj University, Raiganj, West Bengal-733134, India.}
\email{sm05math@gmail.com, sjm@raiganjuniversity.ac.in}
\address{Department of Mathematics, Raiganj University, Raiganj, West Bengal-733134, India.}
\email{naba.iitbmath@gmail.com}
\address{Department of Mathematics, Raiganj University, Raiganj, West Bengal-733134, India.}
\email{debumath07@gmail.com}

\renewcommand{\thefootnote}{}
\footnote{2020 \emph{Mathematics Subject Classification}: 30D35, 30D45.}
\footnote{\emph{Key words and phrases}: Entire functions, derivative, Nevanlinna theory, uniqueness and normal family.}
\footnote{*\emph{Corresponding Author}: Nabadwip Sarkar.}
\renewcommand{\thefootnote}{\arabic{footnote}}
\setcounter{footnote}{0}

\begin{abstract} In the paper, we investigate the uniqueness problem of a power of an entire function that share one value partially with it's linear differential polynomial and obtain a result, which improves several previous results in a large scale. Also in the paper we include some applications of our main result. Moreover, we solve the question raised by Wang and Liu (Value cross-sharing of meromorphic functions, Comput. Methods Funct. Theory (2023). https://doi.org/10.1007/s40315-023-00481-9) for a special case related to value cross-sharing.
\end{abstract}
\thanks{Typeset by \AmS -\LaTeX}
\maketitle
\section{{\bf Introduction}}
We assume that the reader is familiar with standard notation and main results of Nevanlinna Theory (see \cite{WKH1, YY1}). We use notations $\rho(f)$ and $\rho_1(f)$ for the order and the hyper-order of $f$ respectively. Let $f$ and $g$ be two non-constant entire functions and $a\in\mathbb{C}$. If $g-a=0$ whenever $f-a=0$, we write $f=a\Rightarrow g=a$ and we say that $f$ and $g$ share $a$ partially. If $f=a\Rightarrow g=a$ and $g=a\Rightarrow f=a$, we then write $f=a\Leftrightarrow g=a$ and we say that $f$ and $g$ share $a$ IM (ignoring multiplicity). If $f-a$ and $g-a$ have the same zeros with the same multiplicities, we write $f=a\rightleftharpoons g=a$ and we say that $f$ and $g$ share $a$ CM (counting multiplicity).\par

\smallskip
We know that in the study of value sharing, it is often assumed that $f(z)$ and $g(z)$, $f^{(k)}(z)$ and $g^{(k)}(z)$, $f(z)$ and $f^{(k)}(z)$, $f(z)$ and $f (z+c)$ share common values, where $k\in\mathbb{N}$ and $c\in\mathbb{C}\backslash \{0\}$ (see \cite{YY1}).
In 2023, Wang and Liu \cite{WL1} introduced a new type of value sharing which is called the value cross-sharing (see also \cite{GLW1}). For example, $f^n(z)$ and $g^{(1)}$ share common values together with $g^n$ and $f^{(1)}$ share common values, $f^n(z)$ and $g^m(z+c)$ share common values together with $g^n(z)$ and $f^m(z +c)$ share common values, where $m,n\in\mathbb{N}$.

\smallskip
Let $f$ be a non-constant entire function. Then the function $\mathscr{L}_k(f)$ defined by
\[\mathscr{L}_k(f)=f^{(k)}+a_{k-1}f^{(k-1)}+\ldots+a_1f^{(1)}+a_0f=\sideset{}{_{i=0}^k}{\sum}a_if^{(i)},\]
where $a_i\in\mathbb{C}$ $i=0,1,\ldots, k$ is called a linear differential polynomial with constant coefficients generated by $f$. In this paper we always assume that $a_k=1$.\par

\smallskip
The research on the normality of families, concerning meromorphic functions, their derivatives and differential polynomials, is an active and fruitful part in value distribution theory. In this paper, we also use the theory of normal family to deal with the uniqueness problem of a power of an entire function that share one value partially with it's linear differential polynomial.

\smallskip
This paper is organized as follows. In Section 2, we investigate the uniqueness problem for a power of an entire function sharing one value partially with it's linear differential polynomial. We also include some applications of our main result. In Section 3, we solve the question due to Wang and Liu \cite{WL1} for a special case.

\section{\bf{Partial-sharing of entire function with its derivative}}

The shared value problems related to a meromorphic function $f(z)$ and its derivative $f^{(k)}(z)$ have been a more widely studied subtopic of the uniqueness theory of entire and meromorphic functions in the field of complex analysis. Meromorphic solutions of complex differential equations and the uniqueness problems of meromorphic functions sharing value (s) or set with their derivatives have become an area of current interest and a hot direction and the study is based on the Nevanlinna value distribution and theory of normal family.
For this background, we see \cite{CS1}, \cite{GLW1, GY1}, \cite{LM1}, \cite{LY1}-\cite{LUY1}, \cite{MD1}, \cite{MSS1}, \cite{MS1}, \cite{MM1}, \cite{SS1}, \cite{RY1,WY1}, \cite{YZ1}, \cite{ZY2, ZY1}.

\smallskip
The uniqueness problem for a power of an entire function sharing a single value with its derivatives was first observed by Yang and Zhang \cite{YZ1} and they proved the following result.

\begin{theo2.A}\cite[Theorem 4.4]{YZ1} Let $f$ be a non-constant entire function and $n\in\mathbb{N}$ such that $n \geq 7$. If $f^{n}$ and $(f^{n})^{(1)}$ share $1$ CM, then $f(z) = c\exp(\frac{z}{n})$, where $c\in\mathbb{C}\backslash \{0\}$. 
\end{theo2.A}

In 2010, Zhang and Yang \cite{ZY1} improved and generalised Theorem 2.A by considering higher order derivatives and by lowering the power of the entire function. We now state two results of Zhang and Yang \cite{ZY1}.

\begin{theo2.B}\cite[Theorem 2.1]{ZY1} Let $f$ be a non-constant entire function and $k, n\in\mathbb{N}$ such that $n \geq k+1$. If $f^{n}$ and $(f^{n})^{(k)}$ share $1$ CM, then  $f(z) = c\exp(\frac{\lambda}{n}z)$, where $c, \lambda\in\mathbb{C}\backslash \{0\}$ such that $\lambda^{k} = 1$. 
\end{theo2.B} 

\begin{theo2.C}\cite[Theorem 3.1]{ZY1} Let $f$ be a non-constant entire function and $k, n\in\mathbb{N}$ such that $n \geq k+2$. If $f^{n}$ and $(f^{n})^{(k)}$ share $1$ IM, then the conclusion of Theorem 2.B holds. 
\end{theo2.C}

\begin{rem} Following examples ensure that the condition ``$n\geq k+1$ in Theorem 2.B is sharp.
\end{rem}
\begin{exm2A}\cite[Example 2.4]{ZY1} Let $f(z)$ be a non-constant solution of $f^{(1)}(z)-1=\exp (z) (f(z)-1)$. Then $f$ and $f^{(1)}$ share $1$ CM, while $f\not\equiv f^{(1)}$.
\end{exm2A}

\begin{exm} Let $f(z)=2\exp (\frac{z}{2})-1$ and $k=n=1$. It is easy to verify that $f$ and $f^{(1)}$ share $1$ CM, but $f\not\equiv f^{(1)}$.
\end{exm}

In the same paper, Zhang and Yang \cite{ZY1} posed the problem of investigating the validity of Theorem 2.C, for $n \geq k+1$. They could prove Theorem 2.C holds also for $n\geq k+1$ but only for the case when $k=1$. We now recall the result.

\begin{theo2.D}\cite[Theorem 3.2]{ZY1} Let $f$ be a non-constant entire function and $n\in\mathbb{N}\backslash \{1\}$. If $f^{n}$ and $(f^{n})^{(1)}$ share $1$ IM, then $f(z)=c\exp(\frac{1}{n}z)$, where $c\in\mathbb{C}\backslash \{0\}$. 
\end{theo2.D}

In 2023, Majumder et al. \cite{MSS1} further generalised Theorem 2.D by replacing $(f^n)^{(1)}$ by $(f^n)^{(k)}$ and obtained the following result.

\begin{theo2.E}\cite[Theorem 4.1]{MSS1} Let $f$ be a transcendental entire function and $k,n\in\mathbb{N}$. If $f^{n}$ and $(f^{n})^{(k)}$ share $1$ IM, then the conclusion of Theorem 2.B holds.
\end{theo2.E}

\smallskip
The research in this topic has been extended in the following directions:
\begin{enumerate}
\item[(i)] One replaces the shared value by a non-constant function;
\item[(ii)] One replaces sharing CM by partially sharing.
\end{enumerate}

\smallskip
Recently, the present first author and Sarkar \cite{MS1} have extended Theorem 2.D by relaxing the nature of sharing with the idea of ``partially'' sharing value. In fact, they proved the following result.

\begin{theo2.F}\cite[Theorem 1.1]{MS1} Let $f$ be a non-constant entire function and $k,n\in\mathbb{N}$ such that $n\geq k+1$. If $f^{n}=1\Rightarrow (f^{n})^{(k)}=1$, then one of the following cases holds:
\begin{enumerate}
\item[(1)]  $f(z)=c\exp\left(\frac{\lambda }{n}z\right)$, where $c,\lambda \in\mathbb{C}\backslash \{0\}$ such that $\lambda^k=1$.
\item[(2)] $n=2$ and $f(z)=c_0\exp\left(\frac{1}{4}z\right)+c_1$, where $c_0,c_1\in \mathbb{C}\backslash \{0\}$ such that $c_1^2=1$.
\end{enumerate}
\end{theo2.F}

Since a linear differential polynomial is an extension of a derivative, in this paper we consider the problem of extending Theorem 2.F to the linear differential polynomial. We now state our main result.

\begin{theo}\label{t1} Let $f$ be a non-constant entire function and $k, n\in\mathbb{N}$ such that $n \geq k+1$. If $f^{n}=1\Rightarrow \mathscr{L}_k(f^n)=1$, then only one of the following cases holds: 
\begin{enumerate}  
\item[(1)] $f(z)=c\exp (\frac{\lambda}{n}z)$, where $c, \lambda\in\mathbb{C}\backslash \{0\}$ such that $\lambda^{k} = 1$,

\smallskip
\item[(2)] $n=2$ and $f(z)=c_0\exp (\frac{1}{4} z)+c_1$, where $c_0, c_1\in\mathbb{C}\backslash \{0\}$ such that $c_1^2=1$, 

\smallskip
\item[(3)] $(a_0,a_1,\ldots, a_{k-1})\neq(0,0,\ldots,0)$, $n=k+1$, $k\geq 2$ and 
$f(z)=c_0\exp (\lambda z)+c_1$, where $c_0, c_1\in\mathbb{C}\backslash \{0\}$ such that 
where $c_1^{k+1}=1$ and $\lambda^k=(-1)^{k+1}\frac{1}{(k+1)(k+1)!}$,

\smallskip
\item [(4)] $a_0=1$ and $f(z)=A\exp (\frac{\lambda}{n} z)$, where $A, \lambda\in\mathbb{C}\backslash \{0\}$ such that $\lambda^{k-1}+a_{k-1}\lambda^{k-2}+\ldots+a_1=0$,

\smallskip
\item [(5)] $a_0\neq 1$ and $f(z)=A\exp (\frac{\lambda}{n} z)$, where $A, \lambda\in\mathbb{C}\backslash \{0\}$ such that $\lambda^k+a_{k-1}\lambda^{k-1}+\ldots+a_1\lambda+a_0=1$.
\end{enumerate}
\end{theo}

If we add the condition $(f^{k+1})^{(k)}=1\Rightarrow f^{k+1}=1$ in Theorem \ref{t1}, then the conclusions $(2)$ and $(3)$ do not rise. 
Also it is easy to get the following corollary from Theorem \ref{t1}.

\begin{cor}\label{cc1} Let $f$ be a non-constant entire function and $k, n\in\mathbb{N}$ such that $k\geq 2$ and $n \geq k+1$. If $f^{n}=1\Rightarrow (f^{n})^{(k)}=1$, then  $f(z) = c\exp\left(\frac{\lambda}{n}z\right)$, where $c, \lambda\in\mathbb{C}\backslash \{0\}$ such that $\lambda^{k} = 1$.
\end{cor}

\begin{rem} Clearly Corollary \ref{cc1} improves Theorems 2.A-2.F in a large scale.
\end{rem}

Now we exhibit the following example to show that the condition ``$n\geq k+1$'' in Theorem \ref{t1} is sharp.

\begin{exm} Let $\mathscr{L}_k(f)=f^{(k)}$, where $k=n=1$ and $f(z)=\exp (\frac{z}{2})+2\exp(\frac{z}{4})+1$. It is easy to verify that $f=1\Rightarrow f^{(1)}=1$, but $f$ does not satisfy any case of Theorem \ref{t1}.
\end{exm}

\medskip
Observing Theorem \ref{t1}, we see that the problem of the entire function $g$ and its $k$-th derivative sharing one value $a$ is related to the problem of the non-linear differential equation
 \[g^{(1)}(g-\mathscr{L}_k(g))-\varphi g(g-a)=0\]
 having a non-constant entire solution, where $\varphi$ is an entire function. In general, it is difficult to judge whether the differential equation has a non-constant solution. However for the very special case $g=f^n$, where $n\in\mathbb{N}$, we can solve the equation completely.

As the applications of Theorem \ref{t1}, we now present the following results.

\begin{theo}\label{tt1} Let $\varphi$ be an entire function and $k, n\in\mathbb{N}$. Suppose $F$ is a non-constant meromorphic solution of the following non-linear differential equation 
\bea\label{ee1}\label{rm.1} F^{(1)}\left(F-\mathscr{L}_k(F)\right)-\varphi F(F-1)=0,\eea
where $F=f^n$ and $n\geq k+1$. Then the conclusions of Theorem \ref{t1} hold.
\end{theo}

\begin{theo}\label{tt2} Let $\varphi$ be a non-zero entire function and $k, n\in\mathbb{N}$. Suppose $F$ is a non-constant meromorphic solution of (\ref{ee1}), where $F=f^n$ and $n\geq k+1$. Then $n=k+1$ and only one of the following cases holds: 
\begin{enumerate}  
\item[(1)] $n=2$ and $f(z)=c_0\exp (\frac{1}{4} z)+c_1$, where $c_0, c_1\in\mathbb{C}\backslash \{0\}$ such that $c_1^2=1$,

\smallskip
\item[(2)] $(a_0,a_1,\ldots, a_{k-1})\neq(0,0,\ldots,0)$, $n=k+1$, $k\geq 2$ and 
$f(z)=c_0\exp (\lambda z)+c_1$, where $c_0, c_1\in\mathbb{C}\backslash \{0\}$ such that 
where $c_1^{k+1}=1$ and $\lambda^k=(-1)^{k+1}\frac{1}{(k+1)(k+1)!}$.
\end{enumerate}
\end{theo}

Now from Theorems \ref{tt1} and \ref{tt2}, we immediately obtain the following corollary.

\begin{cor}\label{c2} Let $\varphi$ be a non-constant entire function and $k, n\in\mathbb{N}$ such that $n>k+1$. Then the non-linear differential equation (\ref{ee1}), where $F=f^n$ has no solutions. 
\end{cor}

Following example shows that the condition ``$n\geq k+1$'' in Theorem \ref{tt1} is sharp.

\begin{exm} Let $\mathscr{L}_4(f)=f^{(4)}-f^{(3)}-f^{(2)}+f^{(1)}+f$, $n=1$ and $\varphi=0$. Clearly $f(z)=\exp(z)+\exp(-z)$ satisfies the differential equation (\ref{ee1}), but $f$ does not satisfy any case of Theorem \ref{tt1}.
\end{exm}

\medskip
\subsection{{\bf Auxiliary Lemmas}} Let $\mathcal{F}$ be a family of meromorphic functions in a domain $D\subset \mathbb{C}$. We say that $\mathcal{F}$ is normal in $D$ if every sequence $\{f_{n}\}_{n}\subseteq \mathcal{F}$ contains a subsequence which converges spherically and uniformly on the compact subsets of $D$ (see \cite{JS1}). That the limit function is either meromorphic in $D$ or identically equal to $\infty$.

Let $f$ be a meromorphic function. The spherical derivative of $f$ at $z\in\mathbb{C}$ is defined by
\beas f^{\#}(z)=\frac{|f^{(1)}(z)|}{1+|f(z)|^{2}}.\eeas

The following result is the well known Marty's Criterion.
\begin{lem}\label{l2.1}\cite{JS1} A family $\mathcal{F}$ of meromorphic functions on a domain $D$ is normal if and only if for each compact subset $K\subseteq D$, $\exists$ $M\in\mathbb{R}^+$ such that $f^{\#}(z)\leq M$ $\forall$ $f\in\mathcal{F}$ and $z\in K$.
\end{lem} 

We recall the well-known Zalcman's lemma \cite{LZ1}.
\begin{lem}\label{l2.2}\cite[Zalcman's lemma]{LZ1} Let $\mathcal{F}$ be a family of meromorphic functions in the unit disc $\Delta$ such that all zeros of functions in $\mathcal{F}$ have multiplicity greater than or equal to $l$ and all poles of functions in $F$ have multiplicity greater than or equal to $j$ and $\alpha$ be a real number satisfying $-l<\alpha<j$. Then $\mathcal{F}$ is not normal in any neighborhood of $z_{0}\in \Delta$ if and only if there exist
\begin{enumerate}\item[(i)] points $z_{n}\in \Delta$, $z_{n}\rightarrow z_{0}$,
\item[(ii)] positive numbers $\rho_{n}$, $\rho_{n}\rightarrow 0^{+}$ and 
\item[(iii)] functions $f_{n}\in \mathcal{F}$,\end{enumerate}
such that $\rho_{n}^{\alpha} f_{n}(z_{n}+\rho_{n} \zeta)\rightarrow g(\zeta)$ spherically locally uniformly in $\mathbb{C}$, where $g$ is a non-constant meromorphic function. The function $g$ may be taken to satisfy the normalisation $g^{\#}(\zeta)\leq g^{\#}(0)=1 (\zeta \in \mathbb{C})$.
\end{lem}

\smallskip
We may take $-1<\alpha<1$, if there is no restrictions on the zeros and poles of functions in $\mathcal{F}$. If all functions in $\mathcal{F}$ are holomorphic, then we may take $-1<\alpha<\infty$. On the other hand we may choose $-\infty<\alpha<1$ for families of meromorphic functions which do not vanish.

\smallskip
The following result due to Chang and Zalcman \cite{CZ1}.

\begin{lem}\label{l2.3}\cite[Lemma 2]{CZ1} Let $f$ be a non-constant entire function such that $N(r,f)=O(\log r)$ as $r\to\infty$. If $f$ has bounded spherical derivative on $\mathbb{C}$, then $\rho(f)\leq 1$.
\end{lem}

In order to prove Theorem \ref{t1}, we need the following result which is of independent interest.

\begin{lem}\label{l2.4} Let $F=f^{k+1}$ be a non-constant entire function and $k\in\mathbb{N}$. If $F=1\Rightarrow \mathscr{L}_k\left(F\right)=1$, then $\rho(f)\leq 1$.
\end{lem}

\begin{proof}[{\bf Proof of Lemma \ref{l2.4}}]  Let $\mathcal{F}=\{F_{\omega}\}$, where $F_{\omega}(z)=F(\omega+z)$, $z\in\mathbb{C}$. Clearly $\mathcal{F}$ is a family of entire functions defined on $\mathbb{C}$. By assumption, we have $F(\omega+z)=0\Rightarrow F^{(k)}(\omega+z)=0$ and $F(\omega+z)=1\Rightarrow \mathscr{L}_k\left(F(\omega+z)\right)=1$.
Since normality is a local property, it is enough to show that $\mathcal{F}$ is normal at each point $z_0\in \mathbb{C}$.
Suppose on the contrary that $\mathcal{F}$ is not normal at $z_0$. Again since normality is a local property, we may assume that $\mathcal{F}$ is a family of holomorphic functions in a domain $\Delta=\{z: |z-z_0|<1\}$. Without loss of generality,
we assume that $z_0\in\Delta$. Then by Lemma \ref{l2.2}, there exist a sequence of functions $F_n\in\mathcal{F}$, where $F_n(z)=f^{k+1}(\omega_n+z)$, a sequence of complex numbers, $z_n$, $z_n\rightarrow z_0$ and a sequence of positive numbers $\rho_n$, $\rho_n\rightarrow 0$ such that
\bea\label{t1.1} H_n(\zeta)=F_n(z_n+\rho_n \zeta)\rightarrow H(\zeta)\eea
spherically locally uniformly in $\mathbb{C}$, where $H$ is a non-constant entire function such that $H^{\#}(\zeta)\leq 1$, $\forall$ $\zeta\in\mathbb{C}$. Then Lemma \ref{l2.3} gives $\rho(H)\leq 1$. Since all the zeros of $F$ have multiplicity at least $k+1$, Hurwitz's theorem guarantees that all the zeros of $H$ have multiplicity at least $k+1$. Consequently $H^{(k)}\not\equiv 0$.
Also from (\ref{t1.1}), we get
\bea\label{t1.2} H_n^{(i)}(\zeta)=\rho_n^iF_n^{(i)}(z_n+\rho_n \zeta)\rightarrow H^{(i)}(\zeta)\eea
spherically locally uniformly in $\mathbb{C}$ for all $i\in\mathbb{N}$. Clearly from (\ref{t1.1}) and (\ref{t1.2}) we see that
\bea\label{tt1.5} \mathscr{L}_k(H_n(\zeta))=\sideset{}{_{i=0}^k}{\sum} a_i H_n^{(i)}(\zeta)\sideset{}{_{i=0}^k}{\sum}a_i\rho_n^i F_n^{(i)}(z_n+\rho_n \zeta)\rightarrow \mathscr{L}_k(H(\zeta))\eea
spherically locally uniformly in $\mathbb{C}$. Note that 
\[\ol N(r,0;H)\leq \frac{1}{k+1} T(r,H)+S(r,H).\]

If $1$ is a Picard exceptional value of $H$, then we get a contradiction by the second fundamental theorem. Hence $1$ is not a Picard exceptional value of $H$.\par

Let $H(\zeta_0)=1$. Hurwitz's theorem implies the existence of a sequence $\zeta_n\rightarrow \zeta_0$ with $H_n(\zeta_n)=F_n(z_n+\rho_n \zeta_n)=1$. Since $F=1\Rightarrow \mathscr{L}_k(F)=1$, we have $\mathscr{L}_k(F_n(z_n+\rho_n \zeta_n))=1$ and so $\sum_{i=0}^k a_iF_n^{(i)}(z_n+\rho_n \zeta_n)=1$, i.e.,
\bea\label{ttt.1}\sideset{}{_{i=0}^k}{\sum} a_i \rho_n^k F_n^{(i)}(z_n+\rho_n \zeta_n)=\rho_n^k.\eea

Note that $\mathscr{L}_k(H_n(\zeta_n))=\sideset{}{_{i=0}^k}{\sum} a_i\rho_n^i F_n^{(i)}(z_n+\rho_n \zeta_n)$ and so from (\ref{ttt.1}), we get
\bea\label{tt1.6} \mathscr{L}_k(H_n(\zeta_n))=\sideset{}{_{i=0}^{k-1}}{\sum}a_i (1-\rho_n^{k-i})\rho_n^i F_n^{(i)}(z_n+\rho_n \zeta_n)-\rho_n^k.\eea

Consequently from (\ref{t1.2}) and (\ref{tt1.6}), we get
\bea\label{tt1.7} \mathscr{L}_k(H(\zeta_0))=\lim\limits_{n\to \infty} \mathscr{L}_k(H_n(\zeta_n))=\sideset{}{_{i=0}^{k-1}}{\sum}a_i H^{(i)}(\zeta_0).\eea

Again from (\ref{tt1.5}), we have
\bea\label{tt1.8} \mathscr{L}_k(H(\zeta_0))=\lim\limits_{n\to \infty} \mathscr{L}_k(H_n(\zeta_n))=\sideset{}{_{i=0}^k}{\sum}a_i H^{(i)}(\xi_0).\eea

Now from (\ref{tt1.7}) and (\ref{tt1.8}), we obtain
$H^{(k)}(\zeta_0)=0$. This shows that $H=1\Rightarrow  H^{(k)}=0$.

\smallskip
First suppose $0$ is a Picard exceptional value of $H$. Since $H$ is a non-constant entire function having no zeros and $\rho(H)\leq 1$, by Hadamard's factorization theorem, we take $H(\zeta)=A\exp(\lambda \zeta)$, where $A,\lambda\in\mathbb{C}\backslash \{0\}$. Since $H=1\Rightarrow H^{(k)}=0$, we get a contradiction.

\smallskip
Next suppose $0$ is not a Picard exceptional value of $H$. Let $\xi_0$ be a zero of $H$. Now we consider the following family
\beas \mathcal{G}=\left\lbrace G_{n}(\xi): G_{n}(\xi)=\frac{H_n(\xi)}{\sqrt{\rho_n}}=\frac{F_n(z_n+\rho_n \xi)}{\sqrt{\rho_n}}\right\rbrace.\eeas

We prove that the family $\mathcal{G}$ is not normal at $\xi_0$. If not suppose $\mathcal{G}$ is normal at $\xi_0$. Then for a given sequence of functions $\{G_n\}\subseteq \mathcal{G}$, there exist a subsequence of $\{G_n\}$ say itself such that
\bea\label{t1.3} G_n(\xi)=\rho_n^{-\frac{1}{2}}H_n(\xi)\to G(\xi)\eea
or possibly
\bea\label{t1.4} G_n(\xi)\to \infty\eea
spherical uniformly on $\mathbb{C}$ as $n\to\infty$.

Note that $H(\xi_0)=0$ and since $H$ is non-constant, we have $H\not\equiv 0$. Now from (\ref{t1.1}) and Hurwitz's Theorem, there exist $\xi_n$, $\xi_n\to \xi_0$ and $H_n(\xi_n)=0$. Consequently
\bea\label{t1.5} G(\xi_0)=\lim_{n\to\infty}\rho_n^{-\frac{1}{2}}H_{n}(\xi_n)=0.\eea

Now from (\ref{t1.5}), one can easily conclude that (\ref{t1.4}) does not hold. We know that zeros of a non-constant analytic function are isolated. Therefore we can find that there exists some deleted neighborhood $\Delta_{\delta(\xi_0)}=\{\xi: 0<|\xi-\xi_0|<\delta(\xi_0)\}$ of $\xi_0$ such that $H(\zeta)\neq 0$ for all $\xi\in \Delta_{\delta(\xi_0)}$, where
$\delta(\xi_0)\in\mathbb{R}^+$ depends only on $\xi_0$. Therefore for $\xi\in \Delta_{\delta(\xi_0)}$ there exists some positive number $\rho(\xi)$ depending only on $\xi$ such that $|H_n(\xi)|\leq \rho(\xi)$ for sufficiently large values of $n$. Consequently
\[G(\xi)=\lim_{n\to\infty}\rho_n^{-\frac{1}{2}}H_n(\xi)=\infty\]
and so $G(\xi)=\infty$, which contradicts the facts that $G(\xi)\not\equiv \infty$. Hence $\mathcal{G}$ is not normal at $\xi_0$.
Now by Lemma \ref{l2.2}, there exist a sequence of functions $G_n\in\mathcal{G}$, a sequence of complex numbers, $\xi_n$, $\xi_n\rightarrow \xi_0$ and a sequence of positive numbers $\eta_n$, $\eta_n\rightarrow 0$ such that
\bea\label{t1.6} \tilde{G}_n(\xi)=\frac{G_{n}(\xi_n+\eta_n\xi)}{\sqrt{\eta_n}}=\frac{F_n(z_n+\rho_n(\xi_n+\eta_n\xi))}{\sqrt{\rho_n\eta_n}}\rightarrow \tilde G(\xi)\eea
locally uniformly in $\mathbb{C}$, where $\tilde G$ is a non-constant entire function such that 
\bea\label{t1.7}\tilde{G}^{\#}(\xi)\leq \tilde{G}^{\#}(0)=1,\eea $\forall$ $\xi\in\mathbb{C}$. Then using Lemma \ref{l2.3}, we get $\rho(\tilde{G})\leq 1$. Also Hurwitz's theorem guarantees that all the zeros of $\tilde{G}$ have multiplicity at least $k+1$.
Again from the proof of Lemma \ref{l2.2}, we get 
\bea\label{t1.8} \eta_{n}=\frac{1}{G_{n}^{\#}(\xi_{n})}=\frac{1+|G_n(\xi_n)|^2}{|G_n^{(1)}(\xi_n)|}.\eea

Now from (\ref{t1.6}) it is easy to prove that
\bea\label{t1.8a}\eta_n^{\frac{1}{2}}G_n^{(1)}(\xi_n+\eta_n\xi)\to \tilde{G}^{(1)}(\xi)\eea
locally uniformly in $\mathbb{C}$. Note that
\bea\label{t1.9}\tilde{G}^{\#}(0)=\frac{|\tilde{G}^{(1)}(0)|}{1+|\tilde{G}(0)|^2}.\eea

Since $\tilde{G}^{\#}(0)=1$, from (\ref{t1.9}), we get $\tilde{G}^{(1)}(0)\neq 0$. Note that all the zeros of $\tilde{G}$ have multiplicity at least $k+1$. If $\tilde{G}(0)=0$, then obviously $\tilde{G}^{(1)}(0)=0$, which is impossible. Hence $\tilde{G}(0)\neq 0$.
Now from (\ref{t1.6}) and (\ref{t1.8a}), we have 
\bea\label{t1.8b} \frac{\tilde{G}_{n}^{(1)}(\xi)}{\tilde{G}_{n}(\xi)}=\eta_{n}\frac{G_{n}^{(1)}(\xi_{n}+\eta_{n}\xi)}{G_{n}(\xi_{n}+\eta_{n}\xi)}\rightarrow \frac{\tilde{G}^{(1)}(\xi)}{\tilde{G}(\xi)},\eea
locally uniformly in $\mathbb{C}$. Clearly from (\ref{t1.8}) and (\ref{t1.8b}), we see that 
\beas \eta_{n}\left|\frac{G_{n}^{(1)}(\xi_{n})}{G_{n}(\xi_{n})}\right|=\frac{1+|G_{n}(\xi_{n})|^{2}}{|G_{n}^{(1)}(\xi_{n})|}\frac{|G_{n}^{(1)}(\xi_{n})|}{|G_{n}(\xi_{n})|}
=\frac{1+|G_{n}(\xi_{n})|^{2}}{|G_{n}(\xi_{n})|}\rightarrow \left|\frac{\tilde{G}^{(1)}(0)}{\tilde{G}(0)}\right|\neq 0,\infty,\eeas
which implies that $\lim\limits_{n\rightarrow \infty} G_{n}(\xi_{n})\not=0,\infty$ and so from (\ref{t1.6}), we have
\[\tilde{G}_{n}(0)=\eta_{n}^{-\frac{1}{2}}G_{n}(\xi_{n}) \rightarrow \infty.\] 

Again from (\ref{t1.6}), we get $\tilde{G}_{n}(0)\rightarrow \tilde{G}(0)\neq 0,\infty$. So we get a contradiction.\par

Therefore all the foregoing discussion shows that $\mathcal{F}$ is normal at $z_0$. Consequently $\mathcal{F}$ is normal in $\mathbb{C}$. Hence by Lemma \ref{l2.1}, there exists $M > 0$ satisfying $F^{\#}(\omega)=F^{\#}_{\omega}(0)<M$ for all $\omega\in\mathbb{C}$. Consequently by Lemma \ref{l2.3}, we get $\rho(F)\leq 1$ and so $\rho(f)\leq 1$.
 
This completes the proof.
\end{proof}

\begin{lem}\label{l1} \cite[Theorem 4.1]{HKR,NO} Let $f$ be a non-constant entire function such that $\rho(f)\leq 1$ and $k\in\mathbb{N}$. Then $m\big(r,\frac{f^{(k)}}{f}\big)=o(\log r)$ as $r\to \infty$.
\end{lem}

\begin{lem}\label{l2}\cite[Lemma 2]{CCY} Let $f$ be a non-constant meromorphic function and let $a_{n}(\not\equiv 0), 
a_{n-1},\ldots,a_{0}$ be small functions of $f$.  Then $T(r,\sum_{i=0}^{n}a_{i}f^{i})= nT(r,f) + S(r,f).$ 
\end{lem}

\medskip
\begin{proof}[{\bf Proof of Theorem \ref{t1}}] 
Let $F = f^{n}$ and 
\be\label{bm.1} \varphi=F^{(1)}(F -G)/F(F - 1),\ee
where $G=\mathscr{L}_k(F)$. By the given conditions, we have $F= 0\Rightarrow G=0$ and $F=1\Rightarrow G=1$.

Now we consider following two cases.\par

\smallskip
{\bf Case 1.} Let $\varphi \not\equiv 0$. Given conditions, imply that $\varphi$ is an entire function. Also (\ref{bm.1}) yields $m(r,\varphi)=S(r,f)$ and so $T(r,\varphi)=S(r,f)$. Again from (\ref{bm.1}), we get
\bea\label{mb} 1/F =\left(F^{(1)}/(F-1)-F^{(1)}/F\right)\left(1-G/F\right)/\varphi,\eea 
which shows that $m(r,1/F)=S(r,f)$ and so $m(r,1/f)=S(r,f)$. If $n>k+1$, then from (\ref{bm.1}), we deduce that
\[N(r,1/f)\leq N(r,1/\varphi)\leq T(r,\varphi)=S(r,f).\]

Now $m(r,1/f)=S(r,f)$ implies that 
\[T(r,f)=T(r,1/f)+S(r,f)=N(r,1/f)+m(r,1/f)+S(r,f)=S(r,f),\]
which is impossible. Henceforth we assume that $n=k+1$. Therefore $F=f^{k+1}$ and so by Lemma \ref{l2.4}, we get $\rho(F)\leq 1$, i.e., $\rho(f)\leq 1$. 
Now using Lemma \ref{l1} to (\ref{bm.1}), we get $m(r,\varphi)=o(\log r)$ as $r\to\infty$.
Since $N(r,\varphi)=0$, we have $T(r,\varphi)=o(\log r)$ as $r\to\infty$, which shows that $\varphi$ is a constant, say $c\in\mathbb{C}\backslash \{0\}$ and so from (\ref{bm.1}), we have 
\bea\label{bm.1a}\label{ebbm.4} F^{(1)}(F-G)=cF(F-1),\;\;\text{i.e.,}\;\;(k+1)f^{(1)}\left(f^{k+1}-\sideset{}{_{i=0}^k}{\sum}a_i(f^{k+1})^{(i)}\right)=cf\big(f^{k+1}-1\big),\eea
which shows that $f$ is a transcendental entire function. Differentiating (\ref{bm.1a}) once, we get
\bea\label{bm.1aa} F^{(2)}(F-G)+F^{(1)}(F^{(1)}-G^{(1)})=c F^{(1)}(2F-1).\eea

Let $z_{0}$ be a zero of $F$ of multiplicity $m(\geq k+1)$. Then $F(z_{0})=0$ and $F^{(k)}(z_{0})=0$. 
Note that if $m\geq k+2$, then from (\ref{bm.1a}), we get a contradiction. Hence $m=k+1$ and so all the zeros of $F$ have multiplicity exactly $k+1$. Consequently $f$ has only simple zeros.

Let $z_1$ be a simple zero of $f$. Then $z_1$ is a zero of $F$ of multiplicity exactly $k+1$ and so $F(z_1)=0$ and $G(z_1)=0$. Also we see that $G^{(1)}(z_1)=F^{(k+1)}(z_1)$. Now by a routine calculation, one can easily conclude from (\ref{bm.1aa}) that $F^{(k+1)}(z_{1})=c/(k+1)$. 
Consequently 
\bea\label{ebmm12.2} F=0\Rightarrow F^{(k+1)}=c/(k+1).\eea

\smallskip
Since $\varphi=c\in\mathbb{C}\backslash \{0\}$, using Lemma \ref{l1} to (\ref{mb}), we get $m(r,1/F)=o(\log r)$ as $r\to\infty$.
Now from the proof of the Theorem 1.1 \cite{MS1}, we see that $0$ and $1$ are not Picard exceptional values of $f$,
\bea\label{ebbm.2} F^{(k)}=(k+1)!f(f^{(1)})^{k} +k(k-1)(k+1)!f^{2}(f^{(1)})^{k-2}f^{(2)}/4+R_1(f),\eea
\bea\label{ebbm.3} F^{(k + 1)}=(k + 1)!(f^{(1)})^{k + 1} +k(k + 1)(k + 1)!f(f^{(1)})^{k - 1}f^{(2)}/2+R_2(f),\eea  
\be\label{ebbm.3a} F^{(k + 2)}=(f^{k+1})^{(k+2)}=(k + 1)!(k+1)(k+2)(f^{(1)})^{k}f^{(2)}/2+R_3(f),\ee  
where $R_{i}(f)$'s are differential polynomial in $f$, where $i=1,2,3$ and each term of $R_{1}(f)$ contains $f^{m} (3 \leq m \leq k$) as a factor.

\smallskip
Denote by $N(r,0;f^{(1)}\mid F\neq 1)$ the counting function of those zeros of $f^{(1)}$ which are not zeros of $F-1$. We denote by $N(r,1;F\mid\geq 2)$ the counting function of multiple $1$-points of $F$.

\smallskip
Now we divide following two sub-cases.\par

\smallskip
{\bf Sub-case 1.1.} Let $N(r,1;F\mid\geq 2)=0$. Now from (\ref{ebbm.4}), we deduce that $N(r,0;f^{(1)}\mid F\neq 1)=0$.
Since $F=f^{k+1}$ and $N(r,1;F\mid\geq 2)=0$, it follows that $f^{(1)}\neq 0$. Note that $f$ is a transcendental entire function such that $\rho(f)\leq 1$ and $f^{(1)}\neq 0$. So we can take 
\bea\label{ebbm.5} f^{(1)}(z)=d_0\exp \lambda z,\eea
where $d_0,\lambda\in\mathbb{C}\backslash \{0\}$. On integration, we get 
\bea\label{ebbm.6} f(z)=(d_0\exp \lambda z+d_1\lambda)/\lambda,\eea
where $d_1\in\mathbb{C}$. Since $0$ is not a Picard exceptional value of $F=f^{k+1}$, it follows that $d_1\neq 0$.

Let $z_{2}$ be a zero of $f$. Then $f(z_{2})=0$ and $F(z_2)=0$. Therefore from (\ref{ebmm12.2}), we have $F^{(k+1)}(z_{2})=c/(k+1)$ and so
from (\ref{ebbm.3}), we get
\bea\label{ebbm.00}  \big(f^{(1)}(z_{2})\big)^{k+1}=c/(k+1)(k+1)!.\eea

Now from (\ref{ebbm.5})-(\ref{ebbm.00}), we get $d_0^{k+1}\exp (k+1)\lambda z_{2}=(-\lambda d_1)^{k+1}$ and
$d_0^{k+1}\exp (k+1)\lambda z_{2}=c/(k+1)(k+1)!$ and so
\bea\label{ebbm.9} (-\lambda d_1)^{k+1}=c/(k+1)(k+1)!.\eea

Again from (\ref{ebbm.6}), we have 
\bea\label{xx1} f^{k+1}(z)&=&\left(d_0/\lambda\right)^{k+1}\exp (k+1)\lambda z+\;{}^{k+1}C_{1}\left(d_0/\lambda\right)^{k}d_1\exp k\lambda z+\ldots\\&&+\;{}^{k+1}C_{k-1}\left(d_0/\lambda\right)^{2}d_1^{k-1}\exp 2\lambda z+\;{}^{k+1}C_{k}(d_0/\lambda) d_1^{k}\exp \lambda z+d_1^{k+1}\nonumber\eea
and so
\bea\label{xx2} (f^{k+1})^{(i)}(z)&=&\left(d_0/\lambda\right)^{k+1}((k+1)\lambda)^i\exp (k+1)\lambda z+\;{}^{k+1}C_{1}\left(d_0/\lambda\right)^{k}d_1(k\lambda)^i\exp k\lambda z+\ldots\nonumber\\&&+\;{}^{k+1}C_{k-1}\left(d_0/\lambda\right)^{2}d_1^{k-1}(2\lambda)^i\exp 2\lambda z+\;{}^{k+1}C_{k}\frac{d_0}{\lambda} d_1^{k}\lambda^i\exp \lambda z,\eea
where $i\in\mathbb{N}$. Consequently
\bea\label{xx3} \sideset{}{_{i=0}^k}{\sum}a_i(f^{k+1})^{(i)}&=&\left(d_0/\lambda\right)^{k+1}\sideset{}{_{i=0}^k}{\sum}a_i\left((k+1)\lambda\right)^i\exp (k+1)\lambda z\\&&+\;{}^{k+1}C_{1}\left(d_0/\lambda\right)^{k}\sideset{}{_{i=0}^k}{\sum}a_i\left(k\lambda\right)^i d_1\exp k\lambda z+\ldots\nonumber\\
&&+\;{}^{k+1}C_{k-1}\left(d_0/\lambda\right)^{2}\sideset{}{_{i=0}^k}{\sum}a_i(2\lambda)^{i} d_1^{k-1}\exp 2\lambda z\nonumber\\
&&+\frac{(k+1)d_0d_1^{k}}{\lambda}\sideset{}{_{i=0}^k}{\sum}a_i\lambda^{i}\exp \lambda z.\nonumber\eea

Now using (\ref{xx1})-(\ref{xx3}) to (\ref{ebbm.4}), we get 
\bea\label{ebbm.10} &&\left\{\frac{k+1}{\lambda^{k+1}}-\frac{k+1}{\lambda^{k+1}}\sideset{}{_{i=1}^k}{\sum}a_i((k+1)\lambda)^i-\frac{c}{\lambda^{k+2}}\right\}d_{0}^{k+2}\exp (k+2)\lambda z\\
&&+\left\{\frac{(k+1)^{2}}{\lambda^{k}}-\frac{(k+1)^{2}}{\lambda^k}\sideset{}{_{i=1}^k}{\sum}a_i(k\lambda)^i-\frac{c(k+2)}{\lambda^{k+1}}\right\}d_{0}^{k+1}d_{1}\exp (k+1)\lambda z+\ldots\nonumber\\
&&+\left\{\frac{(k+1)^{2}}{\lambda}-\frac{(k+1)^{2}}{\lambda}\sideset{}{_{i=1}^k}{\sum}a_i\lambda^i-c\binom{k+2}{k}\frac{1}{\lambda^{2}}\right\}d_0^2d_1^k\exp 2\lambda z\nonumber\\
&&+\left\{k+1-\frac{c(k+2)}{\lambda}\right\}d_{0}d_{1}^{k+1}\exp \lambda z\nonumber-c d_{1}^{k+2}=\frac{-cd_0}{\lambda}\exp \lambda z-cd_1,\nonumber \eea
which shows that
\bea\label{ebbm.11} d_1^{k+1}=1\;\;\text{and}\;\;\left\{(k+1)-c(k+2)/\lambda\right\}d_{0}d_{1}^{k+1}=-cd_0/\lambda,\;\;
\text{i.e.},\;\; c=\lambda.\eea

Therefore from (\ref{ebbm.9}) and (\ref{ebbm.11}), we have 
\bea\label{ebbm.12} \lambda^k=(-1)^{k+1}/(k+1)(k+1)!.\eea

\smallskip
First we suppose $k=1$. Then from (\ref{ebbm.11}) and (\ref{ebbm.12}), we have respectively
$d_1^2=1$ and $c=\lambda=1/4$. Finally from (\ref{ebbm.6}), we get $f(z)=c_0\exp\left(\frac{1}{4} z\right)+d_1$, where $c_0=4d_0$.

\smallskip
Next we suppose $k\geq 2$. If $(a_0,a_1,\ldots, a_{k-1})=(0,0,\ldots,0)$, then from (\ref{ebbm.11})-(\ref{ebbm.12}), we get
\beas\label{ebbm.13} \frac{k+1}{\lambda^{k+1}}-\frac{k+1}{\lambda^{k+1}}\sideset{}{_{i=1}^k}{\sum}a_i((k+1)\lambda)^i-\frac{c}{\lambda^{k+2}}&=&\frac{1}{\lambda}\left(\frac{k+1}{\lambda^{k}}-(k+1)^{k+1}-\frac{c}{\lambda^{k+1}}\right)\\&=&\frac{k+1}{\lambda}\left((-1)^{k+1}k(k+1)!-(k+1)^k\right)\neq 0\nonumber\eeas
for $k\geq 2$ and so using Lemma \ref{l2} to (\ref{ebbm.10}), we get a contradiction.\par
If $(a_0,a_1,\ldots, a_{k-1})\neq(0,0,\ldots,0)$, then $f(z)=\frac{d_0}{\lambda}\exp \lambda z+d_1$, 
where $d_1^{k+1}=1$ and $\lambda^k=(-1)^{k+1}/(k+1)(k+1)!$.\par

\smallskip
{\bf Sub-case 1.2.} Let $N(r,1;F\mid\geq 2)\neq 0$. Now differentiating (\ref{ebbm.4}) twice, we get
\bea\label{ebbm.14a} &&(k+1)^2f^k (f^{(1)})^2+(k+1)f^{k+1}f^{(2)}-(k+1)f^{(1)}\sideset{}{_{i=0}^k}{\sum} a_i (f^{k+1})^{(i+1)}\\&&-(k+1)f^{(2)}\sideset{}{_{i=0}^k}{\sum} a_i (f^{k+1})^{(i)}
-c(k+2)f^{k+1}f^{(1)}=-cf^{(1)}\nonumber\eea
and
\bea\label{ebbm.14} &&(k+1)^2kf^{k-1}(f^{(1)})^3+3(k+1)^2 f^kf^{(1)}f^{(2)}+(k+1)f^{k+1}f^{(3)}
\\&&-(k+1)f^{(3)}\sideset{}{_{i=1}^k}{\sum}a_i(f^{k+1})^{(i)}-2(k+1)f^{(2)}\sideset{}{_{i=1}^k}{\sum}a_i(f^{k+1})^{(i+1)}\nonumber\\&& -(k+1)f^{(1)}\sideset{}{_{i=1}^k}{\sum}a_i(f^{k+1})^{(i+2)}-c(k+2)(k+1)f^k(f^{(1)})^2
-c(k+2)f^{k+1}f^{(2)}=-cf^{(2)}.\nonumber\eea

Therefore using (\ref{ebbm.2})-(\ref{ebbm.3a}) to (\ref{ebbm.14}), we get
\bea\label{ebbm.15}-(k+1)(k+1)!\big((k^2+3k+6)(f^{(1)})^{k+1}f^{(2)}+2a_{k-1}(f^{(1)})^{k+2}\big)+2R_{4}(f)=-2cf^{(2)},\eea
where $R_{4}(f)$ is a differential polynomial in $f$.

Let $z_2$ be a zero of $f$. Now from (\ref{ebbm.15}), we get
\bea\label{ebbm.17} (k^2+3k+6)(f^{(1)}(z_2))^{k+1}f^{(2)}(z_2)+2a_{k-1}(f^{(1)}(z_2))^{k+2}=\big(2c/(k+1)(k+1)!\big)f^{(2)}(z_2).\eea
 
Therefore from (\ref{ebbm.00}) and (\ref{ebbm.17}), we get $(k^2+3k+4)f^{(2)}(z_2)+2a_{k-1}f^{(1)}(z_2)=0$, which shows that \[f=0\Rightarrow (k^2+3k+4)f^{(2)}+2a_{k-1}f^{(1)}=0.\]

Let 
\bea\label{ebbm.18} H_1=\Big(k^2+3k+4)f^{(2)}+2a_{k-1}f^{(1)}\Big)/f.\eea

We now divide following sub-cases.\par

\smallskip
{\bf Sub-case 1.2.1.} Let $H_1\equiv 0$. Then (\ref{ebbm.18}) gives $(k^2+3k+4)f^{(2)}+2a_{k-1}f^{(1)}\equiv 0.$
Let $\hat a=\frac{2a_{k-1}}{k^2+3k+4}$. On integration, we get 
\[f(z)=\hat A_0\exp(-\hat a z)+\hat B_0,\]
where $\hat A_0, \hat B_0\in\mathbb{C}\backslash \{0\}$. Clearly $f^{(1)}\neq 0$. Note that $F-1=f^{k+1}-1$ and so $F^{(1)}=(k+1)f^kf^{(1)}$. Since $f^{(1)}\neq 0$, we get $N(r,1;F\mid \geq 2)=0$, which is impossible.\par

\smallskip
{\bf Sub-case 1.2.2.} Let $H_1\not\equiv 0$. Since $f$ has only simple zeros and $f=0\Rightarrow (k^2+3k+4)f^{(2)}+2a_{k-1}f^{(1)}=0$, 
from (\ref{ebbm.18}), we can say that $H_1(\not\equiv 0)$ is an entire function. Now using Lemma \ref{l1} to (\ref{ebbm.18}), we get $m(r,H_1)=o(\log r)$ and $T(r,H_1)=o(\log r)$, which shows that $H_1$ is a constant, say $\delta\in\mathbb{C}\backslash \{0\}$.
Then from (\ref{ebbm.18}), we get
\bea\label{ebbm.19} f^{(2)}=\alpha f^{(1)}+\beta f,\;\;
\text{where}\;\; \alpha =-\hat a\;\;\text{and}\;\;\beta =\delta/(k^2+3k+4)\neq 0.\eea 

Differentiating (\ref{ebbm.19}) and using it repeatedly, we have
\bea\label{ebbm.20} f^{(i)}=\alpha_{i-1} f^{(1)}+\beta_{i-1} f,\eea
where $i\geq 2$ and $\alpha_{i-1}, \beta_{i-1}\in\mathbb{C}$.
Now using (\ref{ebbm.20}), we can assume from (\ref{ebbm.2}) that
\bea\label{ebbm.21} \sideset{}{_{i=0}^k}{\sum} a_i (f^{k+1})^{(i)}=\check c_0 f(f^{(1)})^k+\check c_1 f^2(f^{(1)})^{k-1}+\ldots+\check c_k f^kf^{(1)}+\check c_{k+1} f^{k+1},\eea
where $\check c_0(=(k+1)!), \ldots,\check c_{k+1}\in\mathbb{C}$. 
Consequently from (\ref{ebbm.4}) and (\ref{ebbm.21}), we have 
\bea\label{ebbm.22} (k+1)\big(\check c_0(f^{(1)})^{k+1}+\ldots+\check c_k f^{k-1}(f^{(1)})^2+(1-\check c_{k+1})f^kf^{(1)}\big)+cf^{k+1}=c.\eea

Now differentiating (\ref{ebbm.21}) and using (\ref{ebbm.19}) and (\ref{ebbm.20}), one can easily assume that
\bea\label{ebbm.23} \sideset{}{_{i=0}^k}{\sum} a_i (f^{k+1})^{(i+1)}=\check d_0 f(f^{(1)})^k+\check d_1 f^2(f^{(1)})^{k-1}+\ldots+\check d_k f^kf^{(1)}+\check d_{k+1} f^{k+1},\eea
where $\check d_0, \ldots,\check d_{k+1}(=\check c_k\beta)\in\mathbb{C}$. 
Therefore using (\ref{ebbm.19}), (\ref{ebbm.21}) and (\ref{ebbm.23}) to (\ref{ebbm.14a}), we get
\bea\label{ebbm.24}\check e_0 (f^{(1)})^{k+2}+\check e_1 f(f^{(1)})^{k+1}+\ldots+\check e_{k+1} f^{k+1}f^{(1)}+\check e_{k+2} f^{k+2}=cf^{(1)},\eea
where $\check e_0=(k+1)(k+1)!$, 
\[\check e_{k+1}=(k+1)(-\alpha+\check d_{k+1}+\alpha\check c_{k+1}+c(k+2))=(k+1)(-\alpha+\check c_k\beta+\alpha\check c_{k+1}+c(k+2))\]
 and  
\[\check e_{k+2}=(k+1)\beta(\check c_{k+1}-1).\]

Let $z_3$ be a multiple zero of $F-1=f^{k+1}-1$. Then $f(z_3)\neq 0$ and $f^{(1)}(z_3)=0$. Clearly from (\ref{ebbm.24}), we get $\check e_{k+2}=0$. Since $\beta\neq 0$, we have $\check e_{k+1}=1$. Consequently from (\ref{ebbm.24}), we have 
\bea\label{ebbm.25}\check e_0 (f^{(1)})^{k+1}+\check e_1 f(f^{(1)})^{k}+\ldots+\check e_{k+1} f^{k+1}=c. \eea

Let $z_3$ be a multiple zero of $F-1=f^{k+1}-1$. Clearly $f^{k+1}(z_3)=1$ and $f^{(1)}(z_3)=0$. Then from (\ref{ebbm.25}) we get $\check e_{k+1}=c$ and so
\bea\label{ebbm.26} (k+1)(\check c_k\beta+c(k+2))=c.\eea

On the other hand from (\ref{ebbm.22}), we have 
\bea\label{ebbm.27} (k+1)\check c_0(f^{(1)})^{k+1}+\ldots+(k+1)\check c_k f^{k-1}(f^{(1)})^2+cf^{k+1}=c.\eea

Now differentiating (\ref{ebbm.27}) and using (\ref{ebbm.19}), one can easily assume that
\bea\label{ebbm.28} \check f_0(f^{(1)})^{k}+\ldots+(k+1)(2\check c_k\beta+c) f^{k}=0.\eea

If $z_3$ is a multiple zero of $F-1$, then from (\ref{ebbm.28}), one can easily conclude that $2\check c_k\beta+c=0$ and so from (\ref{ebbm.26}), we get a contradiction.\par

\smallskip
{\bf Case 2.} Let $\varphi\equiv 0$. Since $F^{(1)}\not\equiv 0$, it follows that $\mathscr{L}_k(F)\equiv F$, i.e.,
\bea\label{k1} F^{(k)}+a_{k-1}F^{(k-1)}+\ldots+a_1F^{(1)}+(a_0-1)F\equiv 0.\eea

Now we consider following two sub-cases.\par

\smallskip
{\bf Sub-case 2.1.} Let $(a_0,a_1, a_2,\ldots, a_{k-1})=(0,0,\ldots,0)$.
Then from (\ref{k1}), we have $F\equiv F^{(k)}$. The fact $F\equiv F^{(k)}$ implies that $\rho(F)=1$, i.e., $\rho(f)=1$. 
Also $F\equiv F^{(k)}$ implies that $f$ has no zeros. Therefore $f(z) = d\exp\left(\frac{\lambda}{n}z\right)$,
where $d, \lambda \in \mathbb{C}\backslash \{0\}$ such that $\lambda^{k} = 1$.\par

\smallskip
{\bf Sub-case 2.2.} Let $(a_0,a_1, a_2,\ldots, a_{k-1})\neq (0,0,\ldots,0)$. 

\smallskip
First suppose $a_0=1$. Since all the zeros of $F$ have multiplicity at least $k+1$, from (\ref{k1}) we conclude that $F$ has no zeros. Therefore we may assume that $f(z)=A\exp \frac{\lambda}{n} z$, where $A, \lambda\in\mathbb{C}\backslash \{0\}$.
Now from (\ref{k1}), we see that $\lambda^{k-1}+a_{k-1}\lambda^{k-2}+\ldots+a_1=0$. \par

\smallskip
Next suppose $a_0\neq 1$. In this case also $F$ has no zeros and so $f(z)=A\exp \frac{\lambda}{n} z$, where $A, \lambda\in\mathbb{C}\backslash \{0\}$ such that $\lambda^k+a_{k-1}\lambda^{k-1}+\ldots+a_1\lambda+a_0=1$.

This completes the proof.
\end{proof}

\begin{proof}[{\bf Proof of Theorem \ref{tt1}}]
Let $F$ be a non-constant meromorphic solution of the equation (\ref{rm.1}).
Now we consider following two cases.\par

\smallskip
{\bf Case 1.} Let $\varphi\not\equiv 0$. Since $\varphi$ is an entire function, from (\ref{rm.1}), we see that $F$ is a non-constant entire function. Now we prove that $F=1\Rightarrow \mathscr{L}_k(F)=1$. If $1$ is a Picard exceptional value of $F$, then obviously $F=1\Rightarrow \mathscr{L}_k(F)=1$. Next we suppose that $1$ is not a Picard exceptional value of $F$. Let $z_0$ be a zero of $F-1$ of multiplicity $p_0$. Clearly $z_0$ is a zero of $F^{(1)}$ of multiplicity $p_0-1$. Therefore from (\ref{rm.1}), we see that $z_0$ must be a zero of $F-\mathscr{L}_k(F)$ and so $z_0$ is a zero of $\mathscr{L}_k(F)-1$. So $F=1\Rightarrow \mathscr{L}_k(F)=1$. Since $\varphi\not\equiv 0$, we get $F\not\equiv \mathscr{L}_k(F)$. Now proceeding in the same way as done in the proof of Case 1 of Theorem \ref{t1}, we get the conclusions $(2)$ and $(3)$.\par

\smallskip
{\bf Case 2.} Let $\varphi\equiv 0$. Since $F^{(1)}\not\equiv 0$, we get $F\equiv \mathscr{L}_k(F)$. We prove that $F$ is an entire function. For this, let $z_1$ be a pole of $F$ of multiplicity $p_1$. Then $z_1$ is also a pole $\mathscr{L}_k(F)$ of multiplicity $p_1+k$. Since $F\equiv \mathscr{L}_k(F)$, we get a contradiction. Hence $F$ is an entire function. Now proceeding in the same way as done in the proof of Case 2 of Theorem \ref{t1}, we get the conclusions $(1)$, $(4)$ and $(5)$.

This completes the proof.
\end{proof}

\begin{proof}[{\bf Proof of Theorem \ref{tt2}}]
Since $\varphi$ is a non-zero entire function, $\varphi\not\equiv 0$. Now the proof of Theorem \ref{tt2} follows directly from the proof of Theorem \ref{tt1}. So we omit the proof.
\end{proof}

\section{{\bf Value cross-sharing of entire function with its derivative}}

Let $f$ and $g$ be two non-constant entire functions and let $n\geq 1$ be a positive integer. If $f^n(z)f^{(1)}(z)$ and $g^n(z)g^{(1)}(z)$ share $0$ CM, then following two possibilities may occur:
\begin{enumerate}
\item[(i)] $f^n(z)$ and $g^{n}(z)$ share $0$ CM, $f^{(1)}(z)$ and $g^{(1)}(z)$ share $0$ CM;
\item[(ii)] $f^n(z)$ and $g^{(1)}(z)$ share $0$ CM, $g^{n}(z)$ and $f^{(1)}(z)$ share $0$ CM.
\end{enumerate}

Of course, there may exist other possibilities.

\smallskip
For the case $(i)$, Gundersen \cite{GGG1} (see also \cite[Theorems 9.2-9.4]{YY1}) considered the relations between $f(z)$ and $g(z)$. On the other hand for the case $(ii)$, the relations between the transcendental entire functions $f(z)$ and $g(z)$ have been considered by Wang and Liu \cite{WL1}. Now we recall the results
 of Wang and Liu \cite{WL1}.
 
\begin{theo3.A}\cite[Theorem 1.1]{WL1} Let $f$ and $g$ be transcendental entire functions such that $f^n$ and $g^{(1)}$ share $0$ CM, $g^n$ and $f^{(1)}$ share $0$ CM. 

If $n=1$, where $f$ and $g$ are of finite order, then $fg=\beta f^{(1)}g^{(1)}$, where $\beta$ is a non-zero constant. Furthermore, $f$ and $g$ may satisfy one of the following two cases:
\begin{enumerate}
\item[(1)] $f=\gamma g$, $g=\exp(\tau z+\nu)$, where $\tau^2=\frac{1}{\beta}$ is a constant and $\gamma $ is a non-zero
constant;
\item[(2)] $f=\lambda \sin(az + b)$ and $g=\gamma \cos(az + b)$, where $a, b, \lambda , \gamma $ are constants with
$a\lambda \gamma=0$ and $\lambda =i\gamma^2$.
\end{enumerate}

If $n\geq 2$, for any transcendental entire functions $f$ and $g$, then
\beas N(r,0;f)+N(r,0;g)=S(r,f)+S(r,g).\eeas
\end{theo3.A}

If $f$ and $g$ are transcendental entire functions of infinite order and $n=1$, then Theorem 3.A is not true. For example, we consider the functions $f(z)=e^{e^z}$ and $g(z)=e^{-e^z}$ (see \cite[Remark 1.3]{WL1}). However for $f(z)=e^{e^z}$ and $g(z)=e^{-e^z}$ the relation $N(r,0;f)+N(r,0;g)=S(r,f)+S(r,g)$ holds.

\begin{theo3.B}\cite[Theorem 1.2]{WL1} Let $f$ and $g$ be transcendental entire functions and $a$ be a non-zero constant. If $f^n$ and $g^{(1)}$ share $a$ CM, $g^n$ and $f^{(1)}$ share $a$ CM, then $n\leq 2$.
\end{theo3.B}

\smallskip
At the end of the paper, Wang and Liu \cite{WL1} posed the following question:
\begin{ques3.A}\cite[Question 4.2]{WL1} Let $f$ and $g$ be transcendental entire functions. What can we get if $f^n$ and $g^{(1)}$ share a non-zero constant $b$ CM, $g^n$ and $f^{(1)}$ share a non-zero constant $a$ CM, where $a$ and $b$ are any constants?
\end{ques3.A}

In general, it is difficult to solve Question 3.A for two transcendental entire functions $f$ and $g$. However for the very special case $g=f^{(1)}$, where $\rho_1(f)<+\infty$, we can solve Question 3.A completely. Now we state our result.

\begin{theo}\label{4.t1} Let $f$ be transcendental entire function such that $\rho_1(f)<+\infty$ and let $a$ and $b$ be non-zero complex numbers. If $f^n$ and $f^{(2)}$ share $a$ CM, $(f^{(1)})^n$ and $f^{(1)}$ share $b$ CM, then $n=1$.
\end{theo}

\begin{rem} Obviously, the function $f(z)=\sin z+\cos z$ satisfies Theorem \ref{4.t1} when $n=1$. However, we can not find all the entire functions $f$ satisfying Theorem \ref{4.t1}.
\end{rem}

\begin{rem} Theorem \ref{4.t1} improves Theorem 3.B for the special case when $g=f^{(1)}$ and $\rho_1(f)<+\infty$.
\end{rem}

\medskip
We need the following lemmas for the proof of Theorem \ref{4.t1}.
\begin{lem}\label{3.L1} \cite[Theorem 1.64]{YY1} Let $f_1,\ldots, f_n$ be non-constant meromorphic functions and $f_{n+1},\ldots, f_{n+m}$ be meromorphic functions such that $f_k\not\equiv 0\;\;(k=n+1,\ldots, n+m)$ and $\sideset{}{_{i=1}^{m+n}}{\sum} f_i\equiv A\in\mathbb{C}\backslash \{0\}$. If there exists a subset $I\subseteq \mathbb{R}^+$ satisfying $\text{mes}\;I=+\infty$ such that
\beas &&\sideset{}{_{i=1}^{m+n}}{\sum} N(r,0;f_i)+(n+m-1)\sideset{}{_{\substack{i=1, i\neq j}}^{m+n}}{\sum} \ol N(r,f_i)\\&<&(\lambda+o(1))T(r,f_j)\;(r\to \infty, r\in I, j=1,2,\ldots,n),\eeas
where $\lambda<1$, then there exist $t_i\in\{0,1\}\;(i=1,2,\ldots,m)$ such that $\sum_{i=1}^mt_if_{n+i}\equiv A$.
\end{lem}

\begin{lem}\label{3.L2} \cite[Theorem 4]{NO1} Let $f_1,\ldots, f_p$ be $p$ transcendental entire functions such that $\sum_{j=1}^p\lambda_jf_j\equiv 1$, where $\lambda_j\in\mathbb{C}\backslash \{0\}$. Then $\sideset{}{_{j=1}^p}{\sum}\delta (0,f_j)\leq p-1$.
\end{lem}

\begin{lem}\label{3.L4} \cite[Theorem 1.51]{YY1} Suppose that $f_1, f_2,\ldots f_n\;(n\geq 2)$ are meromorphic functions and that $g_1, g_2,\ldots, g_n\;(n\geq 2)$ are entire functions satisfying the following conditions:
\begin{enumerate}
\item[(i)] $f_1e^{g_1}+f_2e^{g_2}+\cdots+f_ne^{g_n}=0$;
\item[(ii)] $g_k-g_j$ are non-constants for all $1\leq j<k\leq n$;
\item[(iii)] For $1\leq j\leq n $ and $1\leq h<k\leq n$, $T(r,f_j)=o\left(T\left(r,e^{g_h-g_k}\right)\right)(r\rightarrow +\infty, r\not\in E)$ where $E\subset [1,\infty)$ is a finite linear measure or finite logarithmic measure.
\end{enumerate}
Then $f_j\equiv 0\;(j= 1, 2,\ldots, n).$
\end{lem}

\begin{lem} \label{3.L5}\cite[Theorem 1.56]{YY1} Let $f_j(j=1, 2,\ldots, n)$ be meromorphic functions that satisfy $\sum_{j=1}^3f_j=1$. If $f_1$ is non-constant and 
\[\sideset{}{_{j=1}^3}{\sum}N(r,0;f_j)+2\sideset{}{_{j=1}^3}{\sum}\ol N(r,f_j)<(\lambda+o(1))T(r),\] where $\lambda <1$ and $T(r)=\max_{1\leq j\leq 3}\{(T(r,f_j)\}$. Then $f_2\equiv 1$ or $f_3\equiv 1$.
\end{lem}

\medskip
\begin{proof}[{\bf Proof of Theorem \ref{4.t1}}]
By the given conditions, $f^n$ and $f^{(2)}$ share $a$ CM and $(f^{(1)})^n$ and $f^{(1)}$ share $b$ CM. Then there exists two entire functions $P$ and $Q$ such that 
\bea\label{xt1.1}\frac{f^n-a}{f^{(2)}-a}=e^P\eea
and  
\bea\label{xt1.2} \frac{(f^{(1)})^n-b}{f^{(1)}-b}=e^Q.\eea

Now from (\ref{xt1.1}) and (\ref{xt1.2}), we get $T(r,\exp(P))=O(T(r,f))$ and $T(r,\exp(Q))=O(T(r,f))$ respectively. Therefore $\rho_1(\exp(P))\leq \rho_1(f)$ and $\rho_1(\exp(Q))\leq \rho_1(f)$. By the given condition, we have $\rho_1(f)<+\infty$. Consequently $\rho(P)=\rho_1(\exp(P))<+\infty$ and $\rho(Q)=\rho_1(\exp(Q))<+\infty$. Hence $P$ and $Q$ are finite order transcendental entire functions.  

Since $f^n$ and $f^{(2)}$ share $a$ CM, by the second fundamental theorem, we get
\beas\label{t11.3} nT(r,f)=T(r,f^n)&\leq& \ol N(r,f^n)+\ol N(r,0;f^n)+\ol N(r,a;f^n)+S(r,f^n)\nonumber\\
&=& \ol N(r,0;f)+\ol N(r,a;f^{(2)})+S(r,f)\nonumber\\
&\leq & T(r,f)+T(r,f^{(2)})+S(r,f)\nonumber\\
&=& m(r,f)+m(r,f^{(2)})+S(r,f)\nonumber\\
&\leq &m(r,f)+m\left(r,\frac{f^{(2)}}{f}\right)+m(r,f)+S(r,f)\nonumber\\
&=&2T(r,f)+S(r,f),\eeas
i.e., $(n-2)T(r,f)=S(r,f)$.
If $n\geq 3$, then we get a contradiction. Hence $n\leq 2$.\par 

Now we want to prove that $n=1$. Suppose on the contrary that $n=2$.

We consider following two cases.\par

\smallskip
{\bf Case 1.} Let $b$ be a Picard exceptional value of $f^{(1)}$. Then we assume that $f^{(1)}=\exp(Q_1)+b$, where $Q_1$ is a non-constant entire function. Since $(f^{(1)})^2$ and $f^{(1)}$ share $b$ CM, $b$ is also a Picard exceptional value of $(f^{(1)})^2$ and so we may assume that $(f^{(1)})^2=\exp(Q_2)+b$, where $Q_2$ is a non-constant entire function. Consequently we have 
\bea\label{xt1.10} \exp(2Q_1)+2b\exp(Q_1)-\exp(Q_2)+b^2-b=0.\eea

\smallskip
Let $b\neq 1$. Then (\ref{xt1.10}) yields $\sum_{j=1}^3f_j=1$, where 
\[f_1=\frac{1}{b^2-b}\exp(2Q_1),\;f_2=\frac{2}{b^2-b}\exp(Q_1)\;\text{and}\;f_3=-\frac{1}{b^2-b}\exp(Q_2).\]

Note that $f_1$, $f_2$ and $f_3$ are non-constants and 
\[\sideset{}{_{j=1}^3}{\sum} N(r,0;f_j)+2\sideset{}{_{j=1}^3}{\sum}\ol N(r,f_j)=o(T(r)),\]
where $T(r)=\max_{1\leq j\leq 3}\{T(r,f_j)\}$. Then by Lemma \ref{3.L5}, we get a contradiction.\par

\smallskip
Let $b=1$. Then from (\ref{xt1.10}), we get $\exp(Q_1)(\exp(Q_1)+2b)=\exp(Q_2)$, which shows that $N(r,-2b;\exp(Q_1))=0$. Then by the second fundamental theorem we get $T(r,\exp(Q_1))=S(r,\exp(Q_1))$ which is impossible.\par

\smallskip
{\bf Case 2.} Let $b$ be not the Picard exceptional value of $f^{(1)}$. 
Let $z_0$ be zero of $f^{(1)}-b$. Since $(f^{(1)})^2$ and $f^{(1)}$ share $b$ CM, we have $f^2(z_0)=b$. Therefore $b^2-b=0$, i.e., $b=1$. Then from (\ref{xt1.2}), we get $f^{(1)}=\exp(Q)-1$. Now differentiating (\ref{xt1.1}) and using $f^{(1)}=\exp(Q)-1$, we get 
\bea\label{xt1.12} f=\frac{((Q^{(1)})^2+P^{(1)}Q^{(1)}+Q^{(2)})\exp(P+Q)-aP^{(1)}\exp(P)}{2(\exp(Q)-1)}.\eea

Again differentiating (\ref{xt1.12}) and using $f^{(1)}=\exp(Q)-1$, we get 
\bea\label{xt1.13} f^{(1)}=\frac{b_1\exp(P+2Q)+b_2\exp(P+Q)+a((P^{(1)})^2+P^{(2)})\exp(P)}{2(\exp(Q)-1)^2},\eea
where 
\beas b_1 &=& P^{(1)}(Q^{(1)})^2+2Q^{(1)}Q^{(2)}+(P^{(1)})^2Q^{(1)}+2P^{(1)}Q^{(2)}+P^{(2)}Q^{(1)}+Q^{(3)}\\
b_2 &=& -\Big((Q^{(1)})^3+P^{(1)}(Q^{(1)})^2+(P^{(1)})^2Q^{(2)}+2P^{(1)}Q^{(2)}+3Q^{(1)}Q^{(2)}+P^{(1)}(Q^{(1)})^2\\&&+P^{(2)}Q^{(1)}+Q^{(3)}+a(P^{(2)}+(P^{(1)})^2)-aP^{(1)}Q^{(1)}\Big).\eeas

Now putting $f^{(1)}=\exp(Q)-1$ into (\ref{xt1.13}), we get 
\bea\label{xt1.13} &&2\exp(3Q)-6\exp(2Q)+6\exp(Q)-a((P^{(1)})^2+P^{(2)})\exp(P)\\&&+b_1\exp(P+2Q)+b_2\exp(P+Q)=1.\nonumber\eea

If $(b_1,b_2)\equiv (0,0)$, then applying Lemma \ref{3.L2} to (\ref{xt1.13}), we get a contradiction. Hence $(b_1,b_2)\not\equiv (0,0)$.
Since $P$ and $Q$ are finite order transcendental entire functions, it follows that $\rho(b_j)<+\infty$ for $j=1,2$.

\smallskip
Suppose $P$ is transcendental while $Q$ is a polynomial. Since $f^{(1)}=\exp(Q)+1$, we have $\rho(f)=\rho(f^{(1)})=\deg(Q)$. On the other hand from (\ref{xt1.1}), we have $T(r,\exp(P))\leq 3T(r,f)+S(r,f)$, which implies that $\rho(\exp(P))\leq \rho(f)$. Therefore we get a contradiction.\par

\smallskip
Suppose $Q$ is transcendental while $P$ is a polynomial. Clearly $T(r,P^{(i)})=S(r,\exp(Q))$ and $T(r,Q^{(i)})=S(r,\exp(Q))$, where $i$ is any positive integer. Consequently $T(r,b_j)=S(r,\exp(Q))$ for $j=1,2$. Now applying Lemma \ref{l2} to (\ref{xt1.13}), we get $3T(r,\exp(Q))=S(r,\exp(Q))$, which is impossible.\par

Hence either both $P$ and $Q$ are transcendental entire functions or both $P$ and $Q$ are polynomials.

Now we divide following two sub-cases.\par

\smallskip
{\bf Sub-case 2.1.} Let $P$ and $Q$ be both transcendental entire functions. Set $h_1=2\exp(3Q)$, $h_2=-6\exp(2Q)$, $h_3=6\exp(Q)$, $h_4=-a((P^{(1)})^2+P^{(2)})\exp(P)$, $h_5=b_1\exp(P+2Q)$ and $h_6=b_2\exp(P+Q)=1$. Then from (\ref{xt1.13}), we get $\sum_{j=1}^6 h_j=1$. Note that 
\[\sideset{}{_{i=0}^6}{\sum} N(r,0;h_i)+3\sideset{}{_{\substack{i=1,i\neq j}}^6}{\sum}\ol N(r,h_i)<(\lambda+o(1))T(r,h_j)\;\;(r\to \infty, r\in I, j=1,2,3,4),\]
$\lambda<1$. Then by Lemma \ref{3.L1}, there exist $t_i\in\{0,1\}, i=1,2$ such that $t_1h_5+t_2g_6=1$, i.e.,
\bea\label{xt1.14} b_1t_1e^{P+2Q}+b_2t_2e^{P+Q}=1.\eea 

Clearly from (\ref{xt1.14}), we get $(b_1t_1,b_2t_3)\not\equiv (0,0)$. Let $b_1t_1\equiv 0$. Then $t_2=1$ and so from (\ref{xt1.14}), we get $b_2e^{P+Q}=1$, which implies $P+Q$ is a polynomial, say $p_1$. Now from (\ref{xt1.13}), we get 
\bea\label{xt1.15} &&2\exp(4Q)-6\exp(3Q)+(6+b_1\exp(p_1))\exp(2Q)+(b_2\exp(p_1)-1)\exp(Q)\nonumber\\&=&a((P^{(1)})^2+P^{(2)})\exp(p_1).\eea
 
Now applying Lemma \ref{l2} to (\ref{xt1.15}), we get $T(r,\exp(Q))=S(r,\exp(Q))$, which is a contradiction.
Let $b_2t_2\equiv 0$. Then $t_1=1$ and so from (\ref{xt1.14}), we get $b_1\exp(P+2Q)=1$, which implies $P+2Q$ is a polynomial, say $p_2$. Now from (\ref{xt1.13}) ,we get 
\beas\label{xt1.16} &&2\exp(5Q)-6\exp(4Q)+6\exp(3Q)+(b_1\exp(p_2)-1)\exp(2Q)+b_2\exp(p_2)\exp(Q)\\&=&a((P^{(1)})^2+P^{(2)})\exp(p_2)\eeas 
and so applying Lemma \ref{l2} we get $T(r,\exp(Q))=S(r,\exp(Q))$, which is a contradiction.\par

Hence $b_1t_1\not\equiv 0$ and $b_2t_2\not\equiv 0$ and so $t_1=t_2=1$. Then (\ref{xt1.14}) yields $\exp(P+Q)(b_1\exp(Q)+b_2)=1$, which shows that $N\big(r,-\frac{b_2}{b_1};\exp(Q))=S(r,\exp(Q))$. Then by the second fundamental theorem for small function (see \cite{KY1}) we get a contradiction.\par
{\bf Sub-case 2.2.} Let $P$ and $Q$ be both polynomials. Let $P(z)=a_sz^s+a_{s-1}z^{s-1}+\cdots+a_0$ and $Q(z)=b_tz^t+b_{t-1}z^{t-1}+\cdots+b_0$. \par

Let $\deg(P)>\deg(Q)$. Then $T(r,\exp(Q))=S(r,\exp(P))$. Also from (\ref{xt1.13}), we get
\beas &&\Big(b_1\exp(2Q)+b_2\exp(Q)-a((P^{(1)})^2+P^{(2)})\Big)\exp(P)\\&&+2\exp(3Q)-6\exp(2Q)+6\exp(Q)-1=0,\eeas
which shows that $b_1\exp(2Q)+b_2\exp(Q)-a((P^{(1)})^2+P^{(2)})=0$ and $2\exp(3Q)-6\exp(2Q)+6\exp(Q)-1=0$. Now by Lemma \ref{l2}, we get $T(r,\exp(Q))=S(r,\exp(Q))$. Therefore $Q$ is a constant, say $c$. So $f'=\exp(c)+1$, i.e., $f$ is a polynomial, which gives a contradiction.\par
If $\deg(P)<\deg(Q)$, then proceeding in the same way as done above, we get a contradiction.
Hence $\deg(P)=\deg(Q)$. 

We divide following sub-cases. \par

\smallskip
{\bf Sub-case 2.2.1.} Let $b_1\not\equiv 0$ and $b_2\equiv 0$. Then both $P$ and $Q$ are non-constants and so from (\ref{xt1.13}), we get 
\bea\label{xt1.18}&& 2\exp(3Q)-6\exp(2Q)+6\exp(Q)-a((P^{(1)})^2+P^{(2)})\exp(P)+b_1\exp(P+2Q)=1.\eea 

Now we divide following two sub-cases.\par

\smallskip
{\bf Sub-case 2.2.1.1.} Let $a_s+2b_s\neq 0$. Since $P$ is non-constant, it follows that $(P^{(1)})^2+P^{(2)}\not\equiv 0$.
Set $F_1=\exp(3Q)$, $F_2=\exp(2Q)$, $F_3=-a((P^{(1)})^2+P^{(2)})\exp(P)$ and $F_4=b_1\exp(P+2Q)$. Note that $\delta(0,F_j)=1$ and so $\sum_{j=1}^4\delta(0,F_j)=4$. Now using Lemma \ref{3.L2} to (\ref{xt1.18}), we get a contradiction.\par

\smallskip
{\bf Sub-case 2.2.1.2.} Let $a_s+2b_s=0$. Then from (\ref{xt1.18}), we get 
\bea\label{t1.19} g_1\exp(3Q)+g_2\exp(2Q)+g_3\exp(Q)+g_4\exp(P)+g_5\exp(0)=0,\eea
where $g_1=2, g_2=-6, g_3=6, g_4=-a((P^{(1)})^2+P^{(2)}), g_5=b_1\exp(P+2Q)-1$. Note that 
\beas c_j=o(T(r,\exp Q))=o(T(r,\exp 2Q))=o(T(r,\exp(P-Q)))=\cdots=o(T(r,\exp(P-2Q))),\eeas
$j=1,\ldots,5$. Now using Lemma \ref{3.L4} to (\ref{t1.19}), we get $c_1=0$, which is impossible.\par

\smallskip
{\bf Sub-case 2.2.2.} Let $b_1\equiv 0$ and $b_2\not\equiv 0$. Then both $P$ and $Q$ are non-constants and so from (\ref{xt1.13}), we get 
\bea\label{xt1.20} &&2\exp(3Q)-6\exp(2Q)+6\exp(Q)-a((P^{(1)})^2+P^{(2)})\exp(P)+b_2\exp(P+Q)=1.\eea 

Now we divide following two sub-cases.\par

\smallskip
{\bf Sub-case 2.2.2.1.} Let $a_s+b_s\neq 0$. Set $F_1=\exp(3Q)$, $F_2=\exp(2Q)$, $F_3=-a((P^{(1)})^2+P^{(2)})\exp(P)$, $F_5=b_1\exp(P+Q)$. Note that $\delta(0,F_j)=1$ and so $\sum_{j=1}^4\delta(0,F_j)=4$. Now using Lemma \ref{3.L2} to (\ref{xt1.20}), we get a contradiction.\par

\smallskip
{\bf Sub-case 2.2.2.2.} Let $a_s+b_s=0$. Then from (\ref{xt1.20}), we have 
\bea\label{t1.21} g_1\exp(3Q)+g_2\exp(2Q)+g_3\exp(Q)+g_4\exp(P)+g_6\exp(0)=0,\eea
where $g_1=2, g_2=-6, g_3=6, g_4=-a((P^{(1)})^2+P^{(2)}), g_6=b_2\exp(P+Q)-1$. 
Now proceeding in the same way as done in Sub-case 2.2.1.2, we get a contradiction.\par

\smallskip
{\bf Sub-case 2.2.3.} Let $b_1\not\equiv 0$ and $b_2\not\equiv 0$. 

If both $P+Q$ and $P+2Q$ are non-constants, then using Lemma \ref{3.L2} to (\ref{xt1.13}), we get a contradiction. Hence at least one of $P+Q$ and $P+2Q$ is a constant.\par

\smallskip
Let $P+Q$ be a constant, say $c_1$. Then from (\ref{xt1.13}), we have 
\beas &&2\exp(4Q)-6\exp(3Q)+(6+b_1\exp(c_1))\exp(2Q)+b_2\exp(c_1)\exp(Q)\\&=&a((P^{(1)})^2+P^{(2)})\exp(c_1)\eeas
and so by Lemma \ref{l2}, we get a contradiction. Again if $P+2Q=c_2$, a constant, then from (\ref{xt1.13}), we get 
$2\exp(5Q)-6\exp(4Q)+6\exp(3Q)+(b_1\exp(c_2)-1)\exp(2Q)+b_2\exp(c_2)\exp(Q)=a((P^{(1)})^2+P^{(2)})\exp(c_2)$
and so by Lemma \ref{l2}, we get a contradiction. 

Hence $n=1$. 

This completes the proof.
\end{proof}

\section{{\bf Statements and declarations}}
\vspace{1.3mm}

\noindent \textbf {Conflict of interest:} The authors declare that there are no conflicts of interest regarding the publication of this paper.\vspace{1.5mm}

\noindent{\bf Funding:} There is no funding received from any organizations for this research work.
\vspace{1.5mm}

\noindent \textbf {Data availability statement:}  Data sharing is not applicable to this article as no database were generated or analyzed during the current study.

\end{document}